\documentclass[12pt]{amsart}
\usepackage{amssymb,amsmath,amsfonts,amsthm,nomencl,mathrsfs} 
\usepackage[arrow, matrix, curve]{xy}
\usepackage[latin1]{inputenc}
\usepackage{a4wide}
\usepackage{color}

\usepackage{hyperref}

\newcommand{\LLL}{\mathscr{L}}

\newcommand{\LL}{\mathsf{L}}
\newcommand{\WW}{\mathsf{W}}
\newcommand{\CC}{\mathsf{C}}
\newcommand{\ICC}{C^{\infty}}

\renewcommand{\dim}{\mathrm{dim}}

\newcommand{\vol}{\mathrm{vol}}

\newcommand{\Id}{{ d}}

\newcommand{\IR}{\mathbb{R}}
\newcommand{\IT}{T}

\newcommand{\Hess}{\operatorname{Hess}}

\newcommand{\ric}{\operatorname{Ric}}

\newcommand{\dist}{\operatorname{dist}}
\newcommand{\bx}{\bar{x}}

\newcommand{\IN}{\mathbb{N}}

\newcommand{\f}{\frac}
\newcommand{\nn}{\nonumber}

\newcommand{\rball}{B^{\IR^{m}}}
\newcommand{\II}{\mathrm{II}}
\renewcommand{\H}{\mathrm{H}}
\newcommand{\inj}{\mathrm{inj}}

\theoremstyle{plain}            % body italics

\newtheorem*{theorem*}{Theorem} % no numbering
\newtheorem*{corollary*}{Corollary}
\newtheorem*{proposition*}{Proposition}

\newtheorem{theorem}{Theorem}[section]

\newtheorem{theoremA}{Theorem}[section] % alpha numbering

\newtheorem{corollaryA}[theoremA]{Corollary}

\newtheorem{Lemma}[theorem]{Lemma}

\newtheorem{Theorem}[theorem]{Theorem}

\theoremstyle{definition}       % body roman

\newtheorem{Remark}[theorem]{Remark}

\newtheorem*{Acknowledgment*}{Acknowledgment}

\begin{document}
\begin{titlepage}\setcounter{page}{1}
\title[ Nonlinear Calder\'on-Zygmund inequalities for maps]{Nonlinear Calder\'on-Zygmund inequalities for maps}

\author[B. G\"uneysu]{Batu G\"uneysu}
\address{Batu G\"uneysu, Institut f\"ur Mathematik, Humboldt-Universit\"at zu Berlin, 12489 Berlin, Germany} \email{gueneysu@math.hu-berlin.de}

\author[S. Pigola]{Stefano Pigola}
\address{Stefano Pigola, Dipartimento di Scienza e Alta Tecnologia - Sezione di Matematica, Università dell'Insubria, 22100 Como, Italy} \email{stefano.pigola@uninsubria.it}

\end{titlepage}

\maketitle 

\begin{abstract} Being motivated by the problem of deducing $\LL^{p}$-bounds on the second fundamental form of an isometric immersion from $\LL^{p}$-bounds on its mean curvature vector field, we prove a nonlinear Calder\'on-Zygmund inequality for maps between complete (possibly noncompact) Riemannian manifolds.   
\end{abstract}

\section*{Introduction and main results}

Let $\psi :(M,g) \to (N,h)$ be an isometric immersion with second fundamental form $\II_{\psi}$ and mean curvature vector field $\H_{\psi} = \mathrm{trace}(\II_{\psi})$. A natural problem in the geometry of submanifolds is to determine how much information can be gained from the knowledge of the mean curvature of the given immersion $\psi$. Assuming that the geometric properties of the ambient space $N$ are sufficiently known, the submanifold geometry of $M$ is encoded into its second fundamental form and this latter, in turn, has strong interplays with the intrinsic geometry of $M$ via Gauss equations. Thus, one is led to investigate how much intrinsic knowledge of $M$ has to be combined with the knowledge of the mean curvature $\H_{\psi}$ in order to deduce information on $\II_{\psi}$. This paper aims at investigating the possibility of deducing $\LL^{p}$-bounds on $\II$ from a corresponding $\LL^{p}$-bound on $\H_{\psi}$. Since $\II_{\psi}$ and $\H_{\psi}$ are nothing but the (generalized) Hessian and Laplacian of the $1$-Lipschitz map $\psi$, it is natural to consider these estimates as a kind of  {\it $\LL^{p}$-Calder\'on-Zygmund inequalities} in the sense of \cite{GP}, but formulated for (special) manifold-valued maps. Thus, we are  led naturally to study possible extensions of the results of \cite{GP} to this more general framework. As a preliminary observation, we point out that when $(N,h) = \IR^{n}$ is the standard Euclidean space, a direct application of \cite[Theorem C]{GP} implies the validity of the following statement:

\begin{theorem*}
Assume that $(M,g)$ is a compact connected $m$-dimensional Riemannian manifold such that
\[
\vol(M) \leq V, \quad \|\ric\|_{\infty}\leq R, \quad r_{\inj}(M) \geq i
\]
for some constants $V,R,i\in (0,\infty)$. Then, for every $p\in (1,\infty)$ and $n \in \IN_{\geq 3}$, there exists a constant $C\in (0,\infty)$ depending only on $p$, $m$, $n$, $V$, $R$, $i$, such that the following estimate holds
\begin{equation}\label {eq-intro0}
 C^{-1} \| \II_{\psi} \|_{p} \leq 1+  \| \H_{\psi} \|_{p} +  \| \dist_{\IR^{n}}(\psi,0) \|_{p},
\end{equation} 
for every isometric immersion $\psi : M \to \IR^{n}$.
\end{theorem*}

A natural consequence of estimate \eqref{eq-intro0} is represented by the next $\LL^{p}$ precompactness conclusion for isometric immersions in the spirit of \cite[Theorem 1.1]{Bre}:

\begin{corollary*}
Let $(M_k,g_{k})$ be a sequence of compact connected Riemannian manifolds satisfying the following conditions:
\[
(a)\,\, \dim(M_{k}) = m, \quad
(b)\,\, \vol(M_{k}) \leq V, \quad
(c)\,\, \|\ric_{M_{k}}\|_{\infty} \leq R, \quad
(d)\,\, r_{\inj}(M_{k}) \geq i>0.
\]
Let $\psi_{k}: M_{k} \to \IR^{n}$ be a sequence of isometric immersions with an $L^{p}$ uniform bound on  their mean curvature, namely:
\[
\| \H_{\psi_{k}} \|_{p} \leq H, 
\]
for some $m<p<+\infty$. Then, there exist an $m$-dimensional Riemannian manifold $M$, an isometric immersion $ \psi : M \to \IR^{n}$ and a sequence of points $y_{k} \in \IR^{n}$ such that, after possibly passing to subsequences, one has
\[
M_{k} \to M, \quad \psi_{k} - y_{k} \to \psi \quad\text{ in the $\CC^{1,\alpha}$-topology.}
\]
 \end{corollary*}

Note  that the existence of an intrinsic  $\CC^{1,\alpha}$ limit of the sequence $M_{k}$ follows from the classical Anderson precompactness result, \cite[Theorem 1.1]{And}. We are grateful to the anonymous referee for having pointed this out. In fact, the intrinsic diameter of the sequence $M_{k}$ is uniformly bounded by a constant $D=D(V,R,i)>0$. This follows e.g. from the fact that, by (c) and (d),  the metric coefficients in harmonic coordinates, hence the volumes, are uniformly controlled on any ball of fixed radius  $r=r(R,i)>0$; see the end of Section \ref{notation}. Thus, by (b), there cannot be too many disjoint balls with uniform radius centered along a minimizing geodesic that realize the diameter.

Now, the key point to prove the Corollary is to show that the $\LL^{p}$-norms of the second fundamental forms of the immersions $\psi_{k}$ are uniformly bounded. Note that, up to translating  each $\psi_{k}(M_{k})$ so to have $0 \in \psi_{k}(M_{k})$, we can always assume
\[
\| \dist_{\IR^{n}}(\psi_{k},0) \|_{p} \leq C
\]
for some uniform constant $C=C(V,R,i,p)>0$. Indeed, since the maps $\psi_{k}$ are $1$-Lipschitz, this is a consequence of the uniform intrinsic diameter bound combined with condition (b). We are in the position to use estimate \eqref{eq-intro0} of the Theorem. \smallskip

In case the ambient space $(N,h)$ is a complete manifold with nontrivial topology, estimates like \eqref{eq-intro0}  are not accessible via methods involving Euclidean targets. The use of Nash\rq{}s embedding theorem does not help here, because, due to its implicit nature, it is not possible to control the resulting extrinsic geometry. New methods, modelled from the very beginning on manifold valued maps need to be implemented. The alluded extension of \eqref{eq-intro0} can be seen as a special case of  the next  {\it nonlinear $\LL^p$-Calder\'on-Zygmund inequality} involving $\LLL$-Lipschitz continuous maps, $0 \leq \LLL \leq +\infty$. As usual, by an $\infty$-Lipschitz map we shall mean a map that is just continuous. Moreover, given positive extended real numbers $a,b \in \IR_{> 0} \cup \{+\infty\}$ we agree to set $a/b =1$ if $a=b=+\infty$. We refer to Section \ref{notation} for more details on notation and definitions.
\begin{theoremA}\label{main}
Let $(M,g)$, $(N,h)$ be connected Riemannian manifolds and set $m:=\dim (M)$, $n:=\dim(N)$. Assume that $M$ is geodesically complete with $\ric_{M} \geq -A$ for some $0 \leq A <+\infty$ and that its $\CC^{1,1/2}$-harmonic radius satisfies $r_{1,1/2}(M) >0$. Assume also that $N$ is geodesically complete with $\CC^{1,1/2}$-harmonic radius $r_{1,1/2}(N) >0$. Then for every $1< p< +\infty$, there exists a constant $C=C(p,m,n,A)>0$, which only depends on the indicated parameters, such that for all $0\leq \LLL \leq +\infty$, all $\LLL$-Lipschitz continuous maps $\psi \in \CC^{2}(M,N)$,  and any $o \in N$, one has 
\begin{align*} 
  C^{-1}\| \Hess ( \psi)\|_{p}  \leq    \| \Delta \psi\|_{p} + r^{-1}\|  d\psi\|_{p}+  r_{1,1/2}(N)^{-1} \| d\psi \|_{2p}^2 +
r^{-2}\|    \dist_{N}(\psi , o) \|_{p},
 \end{align*} 
where we have set
$$
r  =  \min \left(r_{1,1/2}(M), \frac{r_{1,1/2}(N)}{\max(\LLL,1)} , 1\right).
$$
\end{theoremA}

The proof of Theorem \ref{main} will be given in Section \ref{proofs}, where we will also observe that a less precise version of this global inequality can be stated for a slightly larger family of uniformly continuous maps. 
 \vspace{2mm}

Note that in case $\LLL=+\infty$, the statement of Theorem \ref{main} becomes nontrivial only in case $r_{1,1/2}(N)=+\infty$. Note also that an important feature of the above global inequality is that it entails that the possible presence of the $\| d\psi \|_{2p}^2 $ term is due to the possible finiteness of $r_{1,1/2}(N)$. Namely, when $r_{1,1/2}(N)=+\infty$ the $\| d\psi \|_{2p}^2 $ term cancels and the ``traditional'' Calder\'on-Zygmund inequality holds, without imposing any assumption of uniform continuity on the map. Thus Theorem \ref{main} actually recovers, with a new quantitative dependence on the harmonic radius of the source, one of the Euclidean-target results in \cite{GP}. The reason is explained in the next Proposition, that gives rise to the following interesting interpretation: the presence of the $\| d\psi \|_{2p}^2 $ term in the nonlinear Calder\'on-Zygmund inequality measures the curvature on $N$. 

\begin{proposition*}
%\label{prop-rharm-euspace}
Let $(M,g)$ be a complete, non-compact, connected $m$-dimensional Riemannian manifold and assume that there exists some $o \in M$ and some $\alpha\in (0,1)$ such that $r_{1,\alpha}(o)= +\infty$. Then, $(M,g)$ is isometric to the Euclidean $\IR^m$.
\end{proposition*}

Although the result might be known to the experts, we were not able to find any specific reference and, therefore, a complete proof is given in the final appendix.
 \vspace{2mm}

As a simple consequence of Theorem \ref{main} we point out  the following result that, as announced, extends estimate \eqref {eq-intro0} to more general targets. For comparison, recall that the harmonic radii appearing in the statement can be estimated from below by combining a double sided bound on the Ricci tensor  with a positive lower bound of the injectivity radius; see the end of Section \ref{notation} for more details.
\begin{corollaryA}
Let $(M,g)$ be compact and assume that the ambient manifold $(N,h)$ is complete, with $r_{1,1/2}(N)>0$. Then, for every $1<p<+\infty$ and for every $A \geq 0$ such that $\ric_{M}\geq -A$, there exists a constant
$$
C=C\big(p,\dim(M),\dim(N), A\big)>0,
$$
which only depends on the indicated parameters, such that for every isometric immersion $\psi:M\to N$ one has
$$
C^{-1}\| \II_{\psi}\|_{p}  \leq    \| \H_{\psi}\|_{p} + \mathrm{vol}(M)^{1/p}\Big(r^{-1}+  r_{1,1/2}(N)^{-1}  +
r^{-2}  \mathrm{diam}_{N}(\psi (M))\Big)  ,
$$ 
where
$$
r  =  \min \left(r_{1,1/2}(M), r_{1,1/2}(N),1\right).
$$
\end{corollaryA}

The possible applications of global Calder\'on-Zygmund inequalities for maps go beyond the immersion theory. By way of example, a natural problem first investigated by W. Chen and J. Jost in \cite{CJ} concerns with the study of solutions of the \textit{prescribed tension field equation}, i.e., the Poisson equation for a manifold-valued map $\psi: M \to N$. In their seminal paper, Chen and Jost consider the case of compact targets having a non-positive curvature. In particular, they provide a linear(!) $\WW^{2,2}$ a-priori estimate for a map in a give homotopy class by integrating the Bochner formula. Using our completely different approach we are able to extend their $\LL^{2}$ result to much more general situations. There is a prize to pay: our $\LL^{p}$ estimate no more depends on the homotopy class of a given map and, thus, presents a nonlinear term which measure, in some sense, the curvature of the target space; see Theorem \ref{main}. Yet, we believe that our estimate  will still be useful to investigate  existence properties of maps with prescribed ($p$-)tension field.

\section{Some notation}\label{notation}

Given a Riemannian manifold\footnote{If nothing else is said, we understand our manifolds to be smooth and without boundary.} $M=(M,g)$ we denote with $\nabla$ its Levi-Civita connection and with $\mathrm{vol}$ the Riemannian volume measure $ \dist(x,y)$ the geodesic distance and $B_r(x)$ the corresponding open balls. The injectivity radius at $x$ is denoted by $r_{\mathrm{inj}}(x)\in (0,\infty]$. Duals, tensor products and pull-backs of Euclidean vector bundles will be equipped with their canonically given Euclidean metrics, while with the usual abuse of notation these metrics will all be denoted with $|\cdot|$. We understand all our vector bundles and function spaces to be over the field $\IR$. Given another Riemannian manifold $N=(N,h)$, we will sometimes write $\nabla^M$ and $\nabla^N$ etc. in order to distinguish these data. Set $m:=\dim(M)$, $n:=\dim(N)$. For any smooth map $\psi:M\to N$, the vector-bundle $\IT M\otimes \psi^*(\IT N) \to M$ comes equipped with the tensor product covariant derivative 
$$
\nabla^{\psi}=1\otimes (\psi^*\nabla^N) + \nabla^M\otimes 1.
$$ 
In view of
\begin{align*}
%\label{pabbr}
\Id \psi\in \Gamma_{\ICC} (M,\IT^* M\otimes \psi^*(\IT N) ),
\end{align*}
one can consider the generalized Hessian 
\begin{align*}
%\label{pabbr2}
\mathrm{Hess}(\psi):=\nabla^{\psi}\Id \psi \in \Gamma_{\ICC} (M,\IT^* M\otimes\IT^* M\otimes \psi^*(\IT N) ),
\end{align*}
as well as its generalized Laplacian
\begin{align*}
%\label{pabbr3}
\Delta \psi:=\mathrm{tr}_{1,2}(\mathrm{Hess}(\psi)) \in \Gamma_{\ICC} (M, \psi^*(\IT N) ).
\end{align*}
In case $N$ is the Euclidean $\IR^n$, then everything boils down to the usual definitions.

All $\LL^p$-norms are always understood with rexpect to $\mathrm{vol}(dx)$. For example, given a Borel section $\Psi$ in the Euclidean vector bundle $E\to M$ one has 
$$
\left\|\Psi\right\|_p=\big(\int_M |\Psi(x)|^p  \mathrm{vol}(dx)\big)^{1/p},
$$
with the obvious adaption for $p=\infty$. Of course, this notation includes real-valued functions $\Psi:M\to \IR$, which under standard identifications correspond to sections in the trivial line bundle over $M$.

We also record a (slightly modified) definition of the $\CC^{k,\alpha}$-harmonic radius: Given $x\in M$, $k\in\IN_{\geq 0}$, $\alpha\in (0,1)$. Then the $\CC^{k,\alpha}$-harmonic radius $r_{k,\alpha}(x)$ of $M $ at $x$ is defined to be the supremum of all $r>0$ such that $B_{r}(x)$ is relatively compact and admits a centered $\Delta$-harmonic coordinate system $\phi: B_{r}(x)\to  U\subset  \IR^m$ having the following properties: one has
\begin{align}\label{hr1}
(1/2)(\delta^{ij})\leq (g^{ij})\leq 2 (\delta^{ij})\quad\text{ in $B_{r}(x)$ as symmetric bilinear forms}, 
\end{align}
 and for all $i,j\in\{1,\dots, m\}$,
\begin{align}\label{hr2}
& \sum_{\beta\in\IN^m, 1\leq |\beta|\leq k}r^{|\beta|} \sup_{x\rq{}\in B_{r}(x) } |\partial_{\beta}g^{ij}(x\rq{})|+\sum_{\beta\in\IN^m, |\beta|=k}r^{k+\alpha} \sup_{x\rq{},x\rq{}\rq{}\in B_{r}(x), x\rq{}\rq{}\ne x\rq{} }\f{ |\partial_{\beta}g^{ij}(x\rq{})-\partial_{\beta}g^{ij}(x\rq{}\rq{})|}{|x\rq{}-x\rq{}\rq{}|^{\alpha}}\\ \nonumber
&\leq 1.\nn
\end{align}

For every $x\in M$ one has $r_{k,\alpha}(x)\in  (0,\infty]$, and as the function $x\mapsto \min(r_{k,\alpha}(x),1)$ is Lipschitz (cf. Proposition A.1 in \cite{Bru}), it follows that $\inf_U r_{k,\alpha}\in (0,\infty]$ in case $U$ is relatively compact, and  
$$
r_{k,\alpha}(M):=\inf_M r_{k,\alpha}\in [0,\infty]
$$
is the $\CC^{k,\alpha}$-harmonic radius of $M$. We we will be particularly interested in the case $k=1$, where in case $M$ is geodesically complete with $\left\|\mathrm{Ric}\right\|_{\infty}<\infty$ and $r_{\mathrm{inj}}(M)>0$ one has
\begin{align*}
%\label{esto}
r_{k,\alpha}(M)\geq C>0,
\end{align*}
for some constant $C$ depending on $m$, $\left\|\mathrm{Ric}\right\|_{\infty}$ and $r_{\mathrm{inj}}(M)$.

\section{Proof of Theorem \ref{main}}\label{proofs}

The following local result is our main technical tool. As it stands, it does not even require the geodesic completeness of the underlying Riemannian manifolds. It is worth to point out that the result is  new even for real-valued maps, where we can take $R \to +\infty$. Indeed, even in this case, the explicit dependence on the harmonic radius of the source space was never obtained before.

\begin{Theorem}\label{balls}
For all  natural numbers $m, n \geq 2$ and all $1<p<+\infty$ there exists a constant $C=C(n,m,p)>0$, which only depends on the indicated parameters, with the following property: for all
\begin{itemize}
\item [-] connected Riemannian manifolds $M$, $N$ with $m=\dim(M)$, $n=\dim(N)$,
\item [-] $x \in M$, $y\in N$,
\item [-] $r\in(0,\min(r^{M}_{1,1/2}(x),1)/2 ]$, $R\in (0,r^{N}_{1,1/2}(y))$,
\item  [-] $u\in \CC^2(M ,N)$  with $u(B_{r}^{M}(x)) \subset B_{R}^{N}(y)$, 
\end{itemize}
the following estimate holds,
\begin{align*}
C^{-1}\|1_{B^M_{r/2}(x)} \Hess (u) \|_{p} &\leq      \| 1_{B^M_{2r}(x)}\Delta u \|_{p }  + R^{-1}\| 1_{B^M_{2r}(x)} d u\|^2_{2p}+ r^{-2}\| 1_{B^M_{2r}(x)} \dist_N(u,y) \|_{p} \\
&\quad+ r^{-1} \|1_{B^M_{2r}(x)} \Id u \|_{p}.
\end{align*}
\end{Theorem}

%In general, moving from this local result to its global version is not a simple matter of adapting the gluing argument for real-valued maps. In order to carry out this procedure, some choice, based on a combination of the properties of the target space and of the selected family of maps, has to be made.

The proof of Theorem \ref{balls} is based on the following Lemma. The remarkable fact is that it holds for functions in $\WW^{2,q}$ and thus does not need any decay at the boundary. Only this fact makes it possible at all to formulate global Calder\'on-Zygmund inequalities, as in general there is no substitute for compactly supported functions. 

\begin{Lemma}\label{him2} Let $s\in (0,1]$, $q\in (1,\infty)$, and let $P$ be a second order smooth elliptic differential operator of the form $P=\sum_{i,j=1}^ma^{ij}\partial_i \partial_j $ which is defined on the Euclidean ball $B^{\IR^m}_{2s}(0)\subset \IR^m$. Assume that $(a^{ij})\geq 1/2$ as a bilinear form, that for all $i,j$ the function $a^{ij}$ is Lipschitz continuous, and pick $\Lambda>0$, $0<\alpha \leq  1$ such that
$$
\max_{ij}[a^{ij}]_{0,\alpha;B^{\IR^m}_{2s}(0)}\leq  \Lambda s^{-\alpha},\>\>\max_{ij}\left\|a^{ij}\right\|_{\LL^{\infty}(B^{\IR^m}_{2s}(0))}\leq \Lambda,
$$
where
$$
[f]_{0,\alpha;\Omega}=\sup_{x,y\in\Omega, x\ne y}\f{|f(x)-f(y)|}{|x-y|^{\alpha}}
$$
denotes the $\CC^{\alpha}$-seminorm of a function $f:\Omega\to \IR$. Then, there is a constant $C=C(m,\Lambda,\alpha,q)>0$ which only depends on the indicated parameters, such that for every $u\in \WW^{2,q}(B^{\IR^m}_{2s}(0))$ one has 
\begin{align*}
&\|u\|_{\LL^{q}(B^{\IR^m}_{s}(0))}+\Big\|\Big(\sum_{i=1}^m(\partial_i u)^2\Big)^{1/2}\Big\|_{\LL^{q}(B^{\IR^m}_{s}(0))}+\Big\|\Big(\sum_{i,j=1}^m(\partial_i\partial_j u)^2\Big)^{1/2}\Big\|_{\LL^{q}(B^{\IR^m}_{s}(0))}\\
&\leq C \big(  \left\|Pu\right\|_{\LL^{q}(B^{\IR^m}_{2s}(0))}+s^{-2} \left\|u\right\|_{\LL^{q}(B^{\IR^m}_{2s}(0))}\big).
\end{align*}
\end{Lemma}

\begin{proof} For $s=1$ this is precisely the content of Theorem 9.11 in \cite{GT}. The general case  follows from applying this to the scaled operator, resp., function
$$
\tilde{P}:=\sum_{ij}\widetilde{a^{ij}}\partial_i \partial_j,\quad \tilde{u} 
$$  
(where $\tilde{h}(z):=h(zs)$ denotes the scaling operator applied to a function $h:B^{\IR^m}_{2s}(0)\to\IR$) and then scaling back. Note here that  
$$
\tilde{P}\tilde{u}= s^2 \widetilde{Pu},\quad \partial_i\partial_j \tilde{u}= s^2 \widetilde{\partial_i\partial_j u},\quad \partial_i  \tilde{u}= s \widetilde{\partial_i  u},\quad \|\tilde{\cdot}\|_{L^{q}(B^{\IR^m}_{2}(0))}= s^{-m/q}\left\|\cdot\right\|_{L^{q}(B^{\IR^m}_{2s}(0))}
$$
and that
$$
\max_{ij}[\widetilde{a^{ij}}]_{0,\alpha;B^{\IR^m}_{2}(0)}\leq  \Lambda ,
$$
while $(\widetilde{a^{ij}})\geq 1/2$ and $\max_{ij}\left\|\widetilde{a^{ij}}\right\|_{\LL^{\infty}(B^{\IR^m}_{2}(0))}\leq \Lambda$ remain unchanged. 
\end{proof}

Now we can give the proof of Theorem \ref{balls}:

\begin{proof}[Proof of Theorem \ref{balls}] In the sequel, $ X\lesssim Y$ means that there exists a constant $b\in (0,\infty)$, which only depends on $m$ and $n$, such that $X\leq b Y$. In addition,
$$
|A|_{\mathrm{HS}}= \sqrt{\sum_{ij} A_{ij}A _{ij}}
$$
denotes the Hilbert-Schmidt norm of any real-valued matrix $A=(A_{ij})$, and we will use the Einstein sum convention.\\
Fix a $\CC^{1,1/2}$-harmonic coordinate system on $B^{M}_{2r}(x)$ and one on $B^{N}_{R}(y)$. The following pointwise (in)equalities are all understood to hold on $B^{M}_{2r}(x)$. The Hessian of $u$ has the coordinate expression
\[
(\Hess (u))^{\alpha}_{ij} = \Hess (u^{\alpha})_{ij} + \text{ }^{N}\Gamma^{\alpha}_{\beta \gamma}(u)\partial_{i}u^{\beta} \partial_{j}u^{\gamma},
\]
where $i,j,k \in \{1,...,m\}$, $\alpha,\beta,\gamma \in \{1,...,n\}$, $u = (u^{1},...,u^{n})$ and
\[
\Hess(u^{\alpha})_{ij} =  \partial_{i}\partial_{j} u^{\alpha} - \text{ }^{M}\Gamma^{l}_{ij} \partial_{l} u^{\alpha}.
\]
Taking traces we also get the coordinate expression of its Laplacian
\begin{align}\label{helpp}
(\Delta u)^{\alpha} = \Delta u^{\alpha} + \text{ }^{N}\Gamma^{\alpha}_{\beta \gamma}(u) g^{ij } \partial_{i}u^{\beta} \partial_{j}u^{\gamma},
\end{align}
noting that in harmonic coordinates one has
\begin{align}\label{cua133}
\Delta u^{\alpha} = g^{ij}\partial_i\partial_j u^{\alpha}.
\end{align}

Let us compute
\begin{align*}
|\Hess (u) |^{2} &= (\Hess (u))^{\alpha}_{ij} (\Hess (u))^{\beta}_{lk} g^{ik}g^{jl}h_{\alpha \beta}(u)\\
&= \sum_{j,k} (G^{-1}(\Hess (u))^{\alpha})_{kj}(G^{-1}(\Hess (u))^{\beta})_{jk} h_{\alpha \beta}(u),
\end{align*}
where $(\Hess (u))^{\alpha} = ((\Hess (u))^{\alpha}_{ij})$ and $G^{-1}=(g^{ij})$. By the choice of $r,R$, setting
$$
C_1:=C_1(r):=r^{-1},\quad C_2:=C_2(R):=R^{-1},
$$
the following estimates hold:
\begin{align} \nonumber
%\label{cua1}
& \max_{B^{M}_{2r}(x)} |\text{ }^{M}\Gamma^{l}|_{ \mathrm{HS}}  \lesssim C_1 \\\label{cua2}
&\max_{B^{M}_{2r}(x)} |G^{-1}|_{ \mathrm{HS}}  \lesssim 1,\quad [g^{ij}]_{0,1/2,B^M_{2r}(x)}\lesssim r^{-1/2} \\ \nonumber
%\label{cua3}
&(1/2) (\delta^{ij} )  \leq G \leq 2 (\delta^{ij}) 
\end{align}
and
\begin{align*}
%\label{cual1}
& \max_{B_{R}^{N}(y)} |\text{ }^{N}\Gamma^{\alpha }|_{ \mathrm{HS}}  \lesssim C_2\\ 
%\label{cual2}
& \max_{B^{N}_{R}(y)} |H |_{ \mathrm{HS}} \lesssim 1\\ 
%\label{cual3}
&(1/2) (\delta^{\alpha \beta})\leq H \leq 2 (\delta^{\alpha \beta}). 
\end{align*}
where we have set $G=(g_{ij})$, $H = (h_{\alpha \beta})$,  ${}^{M}\Gamma^{l} = ({}^{M}\Gamma^{l}_{ij})$ and ${}^{N}\Gamma^{\alpha} = ({}^{N}\Gamma^{\alpha}_{ij})$.
In particular, using
$$
|du|^{2} = g^{ij}h_{\alpha\beta}(u)\partial_{i}u^{\alpha}\partial_{j}u^{\beta},
$$
we have
\[
  |\partial u|_{ \mathrm{HS}} \lesssim   |du |,
\]
with $\partial u = (\partial_{i}u^{\alpha})$.
Then
\begin{align*}
 |\Hess (u) | &\lesssim  \sqrt{\sum_{j,k,\alpha} (G^{-1}(\Hess (u))^{\alpha})_{kj}^{2}}\\
 &=  \sqrt{\sum_{\alpha} | G^{-1}(\Hess (u))^{\alpha} |^{2}_{\mathrm{HS}}}\\
 &\leq  \sum_{\alpha} | G^{-1}\Hess (u^{\alpha}) + G^{-1}T^{\alpha}|_{\mathrm{HS}}\\
 &\leq \sum_{\alpha} \{ | G^{-1}\Hess (u^{\alpha}) |_{\mathrm{HS}} + | G^{-1}T^{\alpha}|_{\mathrm{HS}} \} \\
 &\leq  |G^{-1}|_{\mathrm{HS}} \sum_{\alpha} \{ | \Hess(u^{\alpha}) |_{ \mathrm{HS}}  +  | T^{\alpha} |_{\mathrm{HS}}\}\\
 &\lesssim  \sum_{\alpha} \left\{ | \Hess(u^{\alpha})|_{\mathrm{HS}} + | T^{\alpha}|_{\mathrm{HS}}\right\}
\end{align*}
with
\[
T^{\alpha}_{ij}: = \text{ }^{N}\Gamma^{\alpha}_{\beta \gamma}(u)\partial_{i}u^{\beta} \partial_{j}u^{\gamma} .
\]
Since
\[
 |T^{\alpha}|_{\mathrm{HS}} \lesssim |\text{ }^{N}\Gamma^{\alpha}(u)|_{\mathrm{HS}} \cdot |\partial u|^{2}_{\mathrm{HS}} \lesssim C_2 |\Id u|^{2},
\]
and
\begin{align*}
| \Hess(u^{\alpha}) |_{\mathrm{HS}} &\lesssim | (\partial_{i} \partial_{j} u^{\alpha})_{ij}  |_{\mathrm{HS}} + \max_l |\text{ }^{M}\Gamma^l|_{\mathrm{HS}} |\partial u|_{\mathrm{HS}}\\
&\lesssim | (\partial_{i} \partial_{j} u^{\alpha})_{ij} |_{\mathrm{HS}} + C_1 | d u|=\big(\sum_{ij} (\partial_{i} \partial_{j}  u^{\alpha})^2\big)^{1/2}  + C_1 | d u|
\end{align*}
from the above we deduce
$$
\|1_{B^M_{r/2}(x)} \Hess (u) \|_{p } \lesssim  \sum_{\alpha} \big\|1_{B^M_{r/2}(x)}\big(\sum_{i,j}(\partial_{i} \partial_{j} u^{\alpha})^2\big)^{1/2}\big\|_{p}
+ C_1 \|1_{B^M_{r/2}(x)} \Id u \|_{p} +  C_2 \|1_{B^M_{r/2}(x)} \Id u \|^{2}_{2p}.
$$
In view of 
$$
B^M_{r/2}(x)\subset B^{\IR^m}_{r/\sqrt{2}}(x)\subset  B^{\IR^m}_{\sqrt{2}r}(x)\subset B^M_{2r}(x)
$$
and using (\ref{cua133}) and (\ref{cua2}), we can apply Lemma \ref{him2} with $s=r / \sqrt{2}$, $P=g^{ij}\partial_i\partial_j$, and use $ \mathrm{vol}( d y)\lesssim d y \lesssim \mathrm{vol}( dy)$ 
to obtain
\begin{align*}
\|1_{B^M_{r/2}(x)} \Hess (u) \|_{p} &\lesssim   C' \sum_{\alpha} \{ \| 1_{B^M_{2r}(x)}\Delta u^{\alpha}\|_{p }  + r^{-2}\| 1_{B^M_{2r}(x)} u^{\alpha}\|_{p}\}\\
&+ C_1 \|1_{B^M_{r/2}(x)} \Id u \|_{p} +  C_2 \| 1_{B^M_{r/2}(x)}\Id u \|^{2}_{2p}
\end{align*}
for some constant $C' = C'(m,n,p )>0$. Using (\ref{helpp}) we get
\begin{align}
|\Delta u^{\alpha}|&\leq |(\Delta u)^{\alpha}|+|\text{ }^{N}\Gamma^{\alpha}_{\beta \gamma}(u) g^{ij } \partial_{i}u^{\beta} \partial_{j}u^{\gamma}|\\
&\lesssim |\Delta u|+ \max_{\alpha}|\text{ }^{N}\Gamma^{\alpha}|_{\mathrm{HS}}|d u|^2\lesssim |\Delta u|+ C_2|d u|^2
\end{align}
and furthermore
\[
|u^{\alpha}|\leq  \sum_{\alpha} \sqrt{(u^{\alpha})^{2}}\leq \sqrt{2} \dist_{N}(u,y),
\]
so that
\begin{align*}
\|1_{B^M_{r/2}(x)} \Hess (u) \|_{p} &\lesssim   C'  \{\| 1_{B^M_{2r}(x)}\Delta u \|_{p }  + C_2\| 1_{B^M_{2r}(x)} d u\|^2_{2p}+ r^{-2}\| 1_{B^M_{2r}(x)} \dist_N(u,y) \|_{p}\}\\
&+ C_1 \|1_{B^M_{r/2}(x)} \Id u \|_{p} +  C_2 \| 1_{B^M_{r/2}(x)}\Id u \|^{2}_{2p},
\end{align*}
completing the proof.
\end{proof}

The proof of Theorem \ref{main} is obtained using Theorem \ref{balls} in two different ways, according to the fact that the center of the balls is taken in a certain partition of the source manifold.

\begin{proof}[Proof of Theorem \ref{main}] 
Let  $\psi \in \CC^{2}(M,N)$ be an $\LLL$-Lipschitz map. We set
$$
\hat{r}:= \frac{1}{16}\min \left(r_{1,1/2}(M), \frac{r_{1,1/2}(N)}{\max(\LLL,1)},1 \right)
$$
and we observe for future use that 
\begin{align}\label{lip-estimate}
&\text{for all $\bx \in M$, $x \in B^{M}_{2\hat{r}}(\bx)$ one has }\\\nn
&\dist_{N}(\psi(x),\psi(\bx))  < \frac{1}{8} r_{1,1/2}(N).
\end{align}
Indeed, if $\LLL<+\infty$ this follows from
$$
\dist_{N}(\psi(x),\psi(\bx)) \leq \LLL \dist_{M}(x,\bx) <  2\LLL \hat{r},
$$
while if $\LLL=+\infty$ we can assume $r_{1,1/2}(N)=+\infty$ and the statement becomes trivial.\\
Now, given a reference point $o \in N$, we define
$$
\Omega_{\psi,o} = \psi^{-1}\left(B^{N}_{\frac{1}{4}r_{1,1/2}(N)}(o)\right)\!,
$$
where as usual $B_{\infty}^N(o):=N$, and we consider the decomposition
$$
M = \Omega_{\psi,o} \cup \Omega_{\psi,o}^{c}.
$$
We are going to elaborate separately the local estimates on the balls $B^{M}_{2\hat{r}}(\bx)$ according to the fact that the center $\bx$ is taken either in $\Omega_{\psi,o}$ or in $\Omega_{\psi,o}^{c}$.\smallskip

Let $\bx \in \Omega_{\psi,o}$. We claim that
$$
\psi(B^M_{2\hat{r}}(\bx))\subset B^N_{\frac{1}{2}r_{1,1/2}(N)}(o). 
$$
Indeed, if $\dist_{M}(x,\bx)<2\hat{r}$, keeping in mind \eqref{lip-estimate}, we have
\begin{align*}
 \dist_{N}(\psi(x),o) &\leq \dist_{N}(\psi(x),\psi(\bx)) + \dist_{N}(\psi(\bx),o)\\
 &\leq \frac{1}{8}r_{1,1/2}(N) + \frac{1}{4}r_{1,1/2}(N) < \frac{1}{2}r_{1,1/2}(N),
\end{align*}
as claimed. If $r_{1,1/2}(N)<\infty$, then applying the local inequality from Theorem \ref{balls} to $r=\hat{r}/4$, $y=o$, and  $R =  r_{1,1/2}(N)/2$ gives that
\begin{align*} 
 C^{-1}\|1_{B^M_{\hat{r}/8}(\bx)} \Hess ( \psi)\|_{p}  \leq   & \quad \|1_{B^M_{\hat{r}}(\bx)}\Delta \psi\|_{p} + \hat{r}^{-1}\| 1_{B^M_{\hat{r}}(\bx)}d\psi\|_{p}+  r_{1,1/2}(N)^{-1} \| 1_{B^M_{\hat{r}}(\bx)}d\psi \|_{2p}^2 \\
&+
\hat{r}^{-2}\| 1_{B^{M}_{\hat{r}}(\bx)}  \dist_{N}(\psi , o) \|_{p},
 \end{align*}
 where $C=C(p,m,n)>0$. In case $r_{1,1/2}(N)=+\infty$, the same inequality may be deduced by applying Theorem \ref{balls} to $r=\hat{r}/4$, $y=o$, and $R>0$ and taking $R\to\infty$.\\ 
Now, we let $\bx \in \Omega_{\psi,o}^{c}$. We claim that
\begin{equation}\label{omega-c}
\text{for every $x \in B^{M}_{2\hat{r}}(\bx)$ one has $\dist_{N}(\psi(x),\psi(\bx)) \leq \dist_{N}(\psi(x),o)$}. 
\end{equation}
Indeed, according to \eqref{lip-estimate},
$$
r_{1,1/2}(N) > 8 \dist_{N}(\psi(x),\psi(\bx)).
$$
Combining this latter with the triangle inequality and the fact that $\bx \in \Omega^{c}_{\psi,o}$, we deduce
\begin{align*}
 \dist_{N}(\psi(x),o) &\geq \dist_{N}(\psi(\bx),o) - \dist_{N}(\psi(x),\psi(\bx)) \\
 &\geq \frac{1}{4}r_{1,1/2}(N) -  \dist_{N}(\psi(x),\psi(\bx)) \\
 &> \frac{1}{4}8 \dist_{N}(\psi(x),\psi(\bx)) -  \dist_{N}(\psi(x),\psi(\bx))\\
 &=  \dist_{N}(\psi(x),\psi(\bx)),
 \end{align*}
as claimed. We shall use \eqref{omega-c} into the local estimate on $B^{M}_{\hat{r}}(\bx)$ in order to get rid of the dependence of the $0$-order term from the center $\bx \in \Omega^{c}_{\psi,o}$. Indeed, by \eqref{lip-estimate},
$$
\psi(B^{M}_{2\hat{r}}(\bx)) \subset B^{N}_{r_{1,1/2}(N)/8}(\psi(\bx)).
$$
Therefore, if $r_{1,1/2}(N)<+\infty$, we can apply the local inequality from Theorem \ref{balls} to $r=\hat{r}/4$, $y=\psi(\bx)$, $R =  r_{1,1/2}(N)/8 $ and obtain
\begin{align*} 
 C^{-1}\|1_{B^M_{\hat{r}/8}(\bx)} \Hess ( \psi)\|_{p}  \leq &\quad   \|1_{B^M_{\hat{r}}(\bx)}\Delta \psi\|_{p} + \hat{r}^{-1}\| 1_{B^M_{\hat{r}}(\bx)}d\psi\|_{p}
+  r_{1,1/2}(N)^{-1} \| 1_{B^M_{\hat{r}}(\bx)}d\psi \|_{2p}^2 \\
&+
\hat{r}^{-2}\| 1_{B^{M}_{\hat{r}}(\bx)}  \dist_{N}(\psi ,  \psi(\bx)) \|_{p},
  \end{align*}
where $C=C(p,m,n)>0$. Again, the $r_{1,1/2}(N)=+\infty$ situation can be treated as above, producing  the same inequality in this case, too.\\
Whence, using \eqref{omega-c}, we conclude the validity of the estimate
\begin{align*} 
 C^{-1}\|1_{B^M_{\hat{r}/8}(x)} \Hess ( \psi)\|_{p}  \leq &\quad   \|1_{B^M_{\hat{r}}(x)}\Delta \psi\|_{p} + \hat{r}^{-1}\| 1_{B^M_{\hat{r}}(x)}d\psi\|_{p}+  r_{1,1/2}(N)^{-1} \| 1_{B^M_{\hat{r}}(\bx)}d\psi \|_{2p}^2\\
& +
\hat{r}^{-2}\| 1_{B^{M}_{\hat{r}}(\bx)}  \dist_{N}(\psi , o) \|_{p}.
 \end{align*}
Summarizing, we have obtained that this estimate holds, for some absolute constant $C=C(p,m,n)>0$, regardless of the location of the center $\bx \in M$.\smallskip

We now pick a sequence of points $\{x_j:j\in \IN \}\subset M$ such that $\{B^{M}_{\hat{r} /8}(x_{j}) :j\in\IN\}$ is a cover of $M$ and $\{B^{M}_{\hat{r} }(x_{j}) :j\in\IN\}$ has an intersection multiplicity $\leq D=D(A)\in (0,\infty)$.  Then summing over $j$ in the last inequality, using monotone convergence and
$$
\sum_j1_{B^M_{\hat{r}/8}(x_j)}\geq 1,\quad  
\sum_j1_{B^M_{   \hat{r}}(x_j) }\leq D,
$$
implies
\begin{align*} 
 C^{-1}\| \Hess ( \psi)\|_{p}  \leq    \| \Delta \psi\|_{p} + \hat{r}^{-1}\|  d\psi\|_{p}+  r_{1,1/2}(N)^{-1} \| d\psi \|_{2p}^2 +
\hat{r}^{-2}\|    \dist_{N}(\psi , o) \|_{p},
 \end{align*}
where $C=C(p,m,n,A)>0$. This completes the proof.
\end{proof}

\begin{Remark}
We see from the above arguments that the only crucial property we have to require on the  map $\psi : M \to N$ is that, given $R>0$ there exists a suitable $r>0$ such that, for every $\bx \in M$, it holds $\psi(B^{M}_{r}(\bx)) \subseteq B^{N}_{R}(\psi(\bx))$. Call such a map an $(r,R)$-uniformly continuous map. Thus, whenever the rays $0<R<r_{1,1/2}(N)/16$ and $0<r<r_{1,1/2}(M)/16$ are fixed, a global inequality of the above type, with an explicit dependence on these rays, can be obtained for the class of $(r,R)$-uniformly continuous maps. 
\end{Remark}

\appendix

\section{On the infinite harmonic radius}
This appendix is devoted to a proof of the following
\begin{proposition*}
Let $(M,g)$ be a complete, non-compact, $m$-dimensional Riemannian manifold and assume that there exists some $o \in M$ and some $\alpha\in (0,1)$ such that $r_{1,\alpha}(o)= +\infty$. Then, $(M,g)$ is isometric to the Euclidean $\IR^m$.
\end{proposition*}
%Proposition \ref{prop-rharm-euspace}, showing that a complete, noncompact Riemannian manifolds with infinity harmonic radius at some fixed point is isometric to the Euclidean flat space.

\begin{proof}
By assumption, there exist a sequence of rays $R_k \to +\infty$ and a corresponding sequence $\varphi_k:B_{R_k}(o) \to \IR^m$ of harmonic coordinates charts centered at $o$ such that conditions  \eqref{hr1} and \eqref{hr2} are satisfied. We consider the corresponding  sequence $(B_{R_k}(o),g,o)$ of pointed Riemannian manifolds  and we show that it has a subsequence that converges in the $\CC^{1}$-topology to $(\IR^m,g_\infty,0)$, where $g_\infty$ is a scalar product with constant coefficients. In particular, $(\IR^m,g_\infty,0)$ is isometric to $\IR^m$. Since the same subsequence of pointed Riemannian manifolds obviously converges in the $\CC^{\infty}$-topology to $(M,g,o)$, we obtain the desired conclusion. Indeed, pointed $\CC^{1}$-convergence implies pointed GH-convergence and, since the limit  metric spaces are proper (they are complete Riemannian manifolds), they must be metrically isometric. But metrically isometric Riemannian manifolds are isometric in the Riemannian sense by the Myers-Steenrod Theorem. \\
Let 
\[
\Omega_k = \varphi_k(B_{R_k}(o)), \quad g_k = (\varphi_{k}^{-1})^\ast g, \quad 0_k = \varphi_k(o)=0.
\]
Observe that, by condition \eqref{hr1},  $\rball_{R_{k}/\sqrt{2}} (0)\Subset \Omega_k$. It follows that $\{\Omega_k\}$ exhausts $\IR^m$. Moreover $\Omega_k \Subset \rball_{\sqrt{2}R_{k}} (0) \Subset \Omega_{k^\prime}$, for every $k^\prime \gg 1$. In particular, each $\Omega_k$ is relatively compact and, up to extracting a subsequence, we can assume that $\Omega_k \Subset \Omega_{k+1}$.
Now, fix $k_0$. According to property \eqref{hr2}, for every $k \gg 1$ we have that the metric coefficients $(g_k)_{ij}$ of $g_k$ are uniformly bounded in $\CC^{1,\alpha}(\bar \Omega_{k_0})$. Since, by Ascoli-Arzel\'{a}, for any domain $D \Subset \IR^m$ and any $0 < \beta < \alpha$ the embedding $\CC^{1,\alpha}(D) \hookrightarrow \CC^{1,\beta}(D)$ is compact, we deduce that a subsequence $(g_{k'})_{ij}$ converges in the $\CC^{1}$-topology on $\Omega_{k_0}$. Now, we let $k_0$ increase to $+\infty$ and we use a diagonal argument to deduce the existence of a $\CC^1$ metric $g_\infty$ on $\IR^m$ such that, in the coordinates of $\IR^m$, a suitable subsequence $\{g_{k''}\}$ $\CC^1$-converges to $g_\infty$ uniformly on compact sets. We stress that $g_\infty$ is actually Riemannian and it is bi-Lip equivalent to the Euclidean metric $g_E$. Indeed, by taking the limit in condition \eqref{hr1} along $g_{k''}$ gives that $2^{-1}\cdot g_{\IR^{m}} \leq g_\infty \leq 2 \cdot g_{\IR^{n}}$. We show that $g_{\infty}$ has constant coefficients. To this end we recall that, by condition \eqref{hr2}, for every $k'' > k''_0$ it holds
\[
\sup_{\Omega_{k^{\prime\prime}_0}} |\partial (g_{k^{\prime\prime}})_{ij}| \leq \frac{1}{R_{k^{\prime\prime}}}.
\]
Letting $k^{\prime\prime} \to +\infty$ shows that the metric coefficients $(g_\infty)_{ij}$ are constant on $\Omega_{k^{\prime\prime}_0}$. Since $k^{\prime\prime}_0$ is arbitrary the claimed property follows. In particular, the covariant differentiation with respect to $g_\infty$ is Euclidean.\\
In conclusion, we have obtained that $(\Omega_{k^{\prime\prime}},g_{k^{\prime\prime}},0_{k^{\prime\prime}})$ $\CC^1$-converges to $(\IR^m,g_\infty,0)$ and this, by definition, means that $(B_{k^{\prime\prime}}(o), g, 0)$ $\CC^1$-converges to the same pointed Riemannian manifold.
\end{proof}

\begin{Acknowledgment*}
 The authors are grateful to Giona Veronelli for suggestions that led to an improvement of the presentation of the paper. The second author is partially supported by INdAM - GNAMPA.
\end{Acknowledgment*}

\end{document}